\documentclass[10pt]{article}
\date{}
\usepackage{amssymb}
\usepackage{amsmath}
\usepackage{amsthm}
\usepackage{graphicx}
\usepackage{indentfirst}
\allowdisplaybreaks[4]
\newtheorem{theorem}{Theorem}

\newtheorem{corollary}[theorem]{Corollary}

\newtheorem{lemma}[theorem]{Lemma}

\numberwithin{equation}{section}
\numberwithin{theorem}{section}
\begin{document}

\title{Gradient estimates for $\Delta_bu+au^{p+1}=0$ on pseudo-Hermitian manifolds}
		\author{ Biqiang Zhao\footnote{
		E-mail addresses: 2306394354@pku.edu.cn}}
		
\maketitle
\setlength{\parindent}{2em}

\begin{abstract}
In this paper, we derive the gradient estimates for the positive solutions of the equation $\Delta_b u+au^{p+1}=0 $ on complete noncompact pseudo-Hermitian manifolds, where $a>0$ and $p\leq 0$ or $a<0$ and $p> 0$ are two constants. As an application, we will obtain a Liouville-type theorem when the manifolds are Sasakian-type with nonnegative pseudo-Hermitian Ricci curvature.
\end{abstract}


\section{Introduction}
\label{1}
   In their classical work, Yau and Cheng \cite{ChY} derived a well-known gradient estimate of positive harmonic functions
   \begin{align}
       \Delta u=0
   \end{align}
   on complete Riemannian manifolds. As a consequence, the Liouville-type theorem holds if $M$ has nonnegative Ricci curvature. It is easy to see that the equation (1.1) can be seen as a special case of 
   \begin{align}
       \Delta u+a(x)u^{p+1}=0,
     \end{align}
    where $p\in \mathbb{R}$. For $a(x)=a<0$ and $p<-1$,  the equation (1.2) on a bounded smooth domain in $\mathbb{R}^n$ is called thin ﬁlm equation, which describes a steady state of the thin film (cf. \cite{GW}). For $p=\frac{4}{n-2}$, the equation (1.2) corresponds to the prescribed scalar curvature equation in $\mathbb{R}^n$. Gidas and Spruck studied the equation (1.2) in \cite{GS} with $a=1 $ and $0\leq p<\frac{4}{n-2}$ when $n>2.$ When $n\geq 4$, Li \cite{Li} studied the gradient estimates of (1.2) and obtained a Liouville theorem when $0\leq p<\frac{2}{n-2}$. Later, Yang \cite{Yyy} derived the gradient estimates of (1.2) when $ p<-1$. Recently, Peng, Wang and Wei \cite{PWW} considered the gradient estimates where  $a>0$ and $p\leq \frac{4}{n}$ or $a<0$ and $p\geq 0$. 
    \par 
     Yau’s estimates have also been generalized in sub-Riemannian geometry. In \cite{CKTL}, Chang et al. derived the gradient estimate for positive pseudoharmonic functions on a complete noncompact pseudo-Hermitian manifold which satisﬁes the CR sub-Laplacian comparison property. In \cite{Ryb}, Ren obtained an explicit gradient estimate of positive eigenfunctions of sub-Laplacian on complete pseudo-Hermitian manifolds. 
     \par
     In this paper. we study the following nonlinear subelliptic equation:
     \begin{align}
         \Delta_bu+au^{p+1}=0
     \end{align}
     on a complete noncompact pseudo-Hermitian $(2m+1)$-manifold. Let $(M^{2m+1},\\HM,
     J,\theta)$ be a complete noncompact pseudo-Hermitian manifold, $ Ric_b,Tor_b $ are defined by 
     \begin{align*}
         Ric_b(X,Y)=R_{\alpha\bar{\beta}}X^\alpha Y^{\bar{\beta}},\quad  Tor_b(X,Y)= \sqrt{-1}( A_{\alpha\beta}X^\alpha Y^\beta -A_{\bar{\alpha}\bar{\beta}}X^{\bar{\alpha}}Y^{\bar{\beta}}).
     \end{align*}
     Here $X=X^\alpha \eta_\alpha,Y=Y^\beta \eta_\beta$ for a frame $\{\eta_\alpha,\eta_{\bar{\alpha}},\xi \}$ of $TM\otimes \mathbb{C}$, $\eta_{\Bar{\alpha}}=\overline{\eta_\alpha} $ and $HM=Re\{T^{1,0}M\oplus T^{0,1}M\}$. $ R_{\gamma\alpha\bar{\beta}}^\delta$  is the pseudo-Hermitian curvature tensor; $R_{\alpha\bar{\beta}}=R_{\gamma\alpha\bar{\beta}}^\gamma$ is the pseudo-Hermitian Ricci curvature tensor; and $A$ is the pseudo-Hermitian torsion tensor.
     \par
     In \cite{HZ}, He and Zhao obtained the gradient estimates of (1.3)  when the solution is bounded below. In \cite{MO}, Ma and Ou obtained a Liouville type theorem on the Heisenberg group $H^m$ when $0<p<\frac{2}{m}$. By studying the methods in \cite{Ryb} and \cite{PWW}, we give the following gradient estimates without the lower bound of the solution.
     \begin{theorem}
         Let $(M^{2m+1},HM,J,\theta)$ be a complete noncompact pseudo-Hermitian manifold with
    \begin{align*}
        Ric_b +2(m-2)Tor_b\geq -k\quad and\quad |A|,|\nabla_b A|\leq k_1\nonumber
    \end{align*}
    for some constants $k,k_1\geq 0$. Assume that $u$ is a smooth positive solution of (1.3).
    \par
    (1) In the case $a>0$ and $p\leq 0$: \\
       there holds true
       \begin{align}
           \frac{|\nabla_b u|^2}{u^2}+au^p+\frac{1}{k_1\epsilon^{-1}+s}\frac{u^2_0}{u^2}\leq  \frac{\lambda_s}{\frac{1}{m}-\epsilon}+\frac{\lambda_s^2}{4(1-\epsilon m)( k_1\epsilon^{-1}+s)}
       \end{align}
        for any $s>0,\epsilon\in (0,\frac{1}{6m})$ and 
        \begin{align*}
            \lambda_s=2k+2\epsilon+16s+16k_1\epsilon^{-1}.
        \end{align*}
     (2)    In the case $a<0$ and $p> 0$:
     \par
      (i)  there holds true in the case $p\geq \frac{1}{m}$
      \begin{align}
           \frac{|\nabla_b u|^2}{u^2}+\frac{1}{k_1\epsilon^{-1}+s}\frac{u^2_0}{u^2}\leq  \frac{\lambda_s}{\frac{1}{m}-\epsilon}+\frac{\lambda^2_s}{4(1-\epsilon m)( k_1\epsilon^{-1}+s)}
       \end{align}
        for any $s>0,\epsilon\in (0,\frac{1}{6m})$ and 
        \begin{align*}
            \lambda_s=2k+2\epsilon+16s+16k_1\epsilon^{-1}.
        \end{align*}
        \par
        (ii) there holds true in the case $0<p< \frac{1}{m}$
        \begin{align}
           \frac{|\nabla_b u|^2}{u^2}+\frac{1}{k_1\epsilon^{-1}+s}\frac{u^2_0}{u^2}\leq  \frac{\lambda_s}{\frac{1}{m}-m(\frac{1}{m}-p)^2-\epsilon}+\frac{\lambda^2_s}{4(1-\epsilon m-(1-mp)^2)( k_1\epsilon^{-1}+s)}
       \end{align}
        for any $s>0,\epsilon\in (0,min\{\frac{1}{6m},\frac{1}{m}-m(\frac{1}{m}-p)^2\})$ and 
        \begin{align*}
            \lambda_s=2k+2\epsilon+16s+16k_1\epsilon^{-1}.
        \end{align*}
     \end{theorem}
     If $M$ is a Sasakian manifold, then the term $\epsilon^{-1}$ vanish. Let $\epsilon\xrightarrow{} 0$, we can obtain the following corollary.
     \begin{corollary}
         Let $(M^{2m+1},HM,J,\theta)$ be a complete noncompact Sasakian manifold with
    \begin{align*}
        Ric_b\geq -k\nonumber
    \end{align*}
    for some constant $k\geq 0$. Assume that $u$ is a smooth positive solution of (1.3).
    \par
    (1) In the case $a>0$ and $p\leq 0$: \\
       there holds true
       \begin{align*}
           \frac{|\nabla_b u|^2}{u^2}+au^p+\frac{1}{s}\frac{u^2_0}{u^2}\leq 2m(2k+8s)+(k+8s)^2 s^{-1}
       \end{align*}
        for any $s>0$.
        \\
     (2)    In the case $a<0$ and $p> 0$:
     \par
      (i)  there holds true in the case $p\geq \frac{1}{m}$
      \begin{align*}
           \frac{|\nabla_b u|^2}{u^2}+\frac{1}{s}\frac{u^2_0}{u^2}\leq  2m(2k+8s)+(k+8s)^2 s^{-1}
       \end{align*}
        for any $s>0$.
        \par
        (ii) there holds true in the case $0<p< \frac{1}{m}$
        \begin{align*}
           \frac{|\nabla_b u|^2}{u^2}+\frac{1}{s}\frac{u^2_0}{u^2}\leq  \frac{2k+16s}{\frac{1}{m}-m(\frac{1}{m}-p)^2}+\frac{(k+8s)^2 s^{-1}}{1-(1-mp)^2}
       \end{align*}
        for any $s>0$.
     \end{corollary}
      As an application, we obtain the following Liouville-type theorem.
      \begin{corollary}
           Let $(M^{2m+1},HM,J,\theta)$ be a complete noncompact Sasakian manifold with nonnegative pseudo-Hermitian Ricci curvature. Then equation (1.3) with $a>0$ and $p\leq 0$ or $a<0$ and $p> 0$ does not admit any positive solution.
      \end{corollary}
      Throughout this paper, we use Einstein convention, i.e. repeated index implies summation.

      \section{Preliminaries}
      In this section, we introduce some basic materials in pseudo-Hermitian geometry (see \cite{DT,We} for more details) and gives some formulas.
      \par
      A real $2m + 1$ dimensional orientable $C^\infty$ manifold $M^{2m+1}$ is said to be a CR manifold if there exists a rank $m$ complex subbundle $T^{1,0}M $ of $TM\otimes \mathbb{C}$ satisfying 
\begin{equation}
T^{1,0}M\cap T^{0,1}M=\{0\} ,\ \  [\Gamma(T^{1,0}M),\Gamma(T^{1,0}M)]\subseteq \Gamma(T^{1,0}M),   
\end{equation}
where $T^{0,1}M=\overline{T^{1,0}M}$. Equivalently, the CR structure can be described by the real bundle $HM=Re\{T^{1,0}M\oplus T^{0,1}M\}$ and an almost complex structure $J$ on $HM$, where $J(X+\overline{X})=\sqrt{-1}(X-\overline{X})$ for any $X\in T^{1,0}M$. A global nowhere vanishing 1-form $\theta$ is called a pseudo-Hermitian structure on $M$. The orientability of $M$ ensures that such $\theta$ always exists. The Levi form $L_\theta$ of a pseudo-Hermitian structure $\theta$ is given by
\begin{eqnarray}
L_\theta(X,Y)=d\theta(X,JY) \nonumber
\end{eqnarray}
for any $X,Y\in HM$. A CR manifold $(M, HM, J,\theta)$ is said to be strictly pseudoconvex if $L_\theta$ is positive definite on $HM$.  Then the quadruple $(M, HM, J, \theta)$ is called a pseudo-Hermitian manifold. In particular, there exits a Reeb vector file $\xi$ on $(M, HM, J, \theta)$ such that $\theta(\xi)=1,\ d\theta(\xi,\cdot)=0,  $ which induces a direct sum decomposition on $TM=HM \oplus \mathbb{R}\xi$.  This allows us to define a Riemannian metric
\begin{eqnarray}
g_\theta=L_\theta+\theta\otimes \theta .\nonumber
\end{eqnarray} 
       \par
       It is known that there exists a canonical connection $\nabla$ on the pseudo-Hermitian manifold, called the Tanaka-Webster connection, such that $\nabla$ preserving the horizontal bundle, the CR structure and the Webster metric. Moreover, its torsion satisfies
\begin{align*}
    T_\nabla(X,Y)=2d\theta(X,Y)\xi\ ,\quad\ T_\nabla(\xi,JX)+JT_\nabla(\xi,X)=0,
\end{align*}
for $X,Y\in TM$. Here we extend $J$ to an endomorphism of $TM$ by requiring that $J\xi=0$. The pseudo-Hermitian torsion of $\nabla$ is defined by
$\tau(X)=T_\nabla (\xi,X)$ for any $X\in TM$. Set $A(X,Y)=T_\nabla(\tau(X),Y)$ for any $X,Y\in TM$. We say that $M$ is Sasakian if $\tau=0$. 
\par
    Suppose that $\{\eta_\alpha\}_{\alpha=1}^m$ is a local unitary frame field of $T^{1,0}M$ and $\{ \theta^1,\cdots,\theta^m\}$ be the dual frame field of $\{\eta_\alpha \}_{\alpha=1}^m$. Then we have the following structure equations for the Tanaka-Webster connection $\nabla$ (cf. \cite{We}):
\begin{eqnarray}
d\theta&=&2\sqrt{-1}\theta^\alpha \wedge \theta^{\bar{\alpha}}, \nonumber\\
d\theta^{\alpha}&=&\theta^\beta\wedge \theta^\alpha_\beta+A_{\bar{\alpha}\bar{\beta}}\theta\wedge\theta^\beta,\\
d\theta^\alpha_\beta&=&\theta^\gamma_\beta\wedge\theta^\alpha_\gamma+\Pi^\alpha_\beta \nonumber
\end{eqnarray}
with
\begin{eqnarray}
\Pi^\alpha_\beta=2\sqrt{-1}(\theta^\alpha\wedge\tau^{\bar{\beta}}-\tau^\alpha \wedge \theta^{\bar{\beta}})+R^\alpha_{\beta\lambda\bar{\mu}}\theta^\lambda\wedge\theta^{\bar{\mu}}+W^\alpha_{\beta\bar{\gamma}}\theta\wedge\theta^{\bar{\gamma}}-W^\alpha_{\beta\gamma}\theta\wedge\theta^{{\gamma}}, \nonumber
\end{eqnarray}   
where $W^\alpha_{\beta\bar{\gamma}}=A^\alpha_{\bar{\gamma}, \beta},\ W^\alpha_{\beta\gamma}=A^{\bar{\gamma}}_{\beta,\bar{\alpha}}$ are the covariant derivatives of $A$, and $R_{\beta\lambda\bar{\mu}}^\alpha$ are the components of curvature tensor. 
\par
  For a smooth function $v$, its gradient $\nabla v$ can be expressed as
\begin{eqnarray}
\nabla v=v_0\xi+ v_{\bar{\alpha}}\eta_\alpha+v_\alpha \eta_{\bar{\alpha}} ,\nonumber
\end{eqnarray}
where $v_0=\xi(v),v_\alpha=\eta_\alpha(v),v_{\bar{\alpha}}=\eta_{\bar{\alpha}}(v)$. Then the horizontal gradient of $v$ is defined by
\begin{eqnarray}
\nabla_b v=v_{\bar{\alpha}}\eta_\alpha+v_\alpha \eta_{\bar{\alpha}} .\nonumber
\end{eqnarray}
The sub-Laplacian of a smooth function $v$ is defined by 
\begin{align*}
    \Delta_b v= trace_{G_\theta}\nabla_bd_b v,
\end{align*}
where $ \nabla_bd_b v$ is the restriction of $\nabla d v$ on $HM\times HM$. In particular
\begin{align*}
    |\nabla_b v|^2=2v_\alpha v_{\bar{\alpha}}, \quad  | \nabla_bd_b v|^2=2(v_{\alpha\beta}v_{\bar{\alpha}\bar{\beta}}+v_{\alpha\bar{\beta}}v_{\bar{\alpha}\beta}) .
\end{align*}
Let us recall the following CR Bochner formula.
 \begin{lemma} (cf. \cite{Ryb})
     Let $(M^{2m+1},HM,J,\theta)$ be a pseudo-Hermitian manifold with
     \begin{align*}
        Ric_b +2(m-2)Tor_b\geq -k\quad and\quad |A|,|\nabla_b A|\leq k_1 ,\nonumber
    \end{align*}
    for some constants $k,k_1\geq 0$. Then for any smooth function $v$ and any $\epsilon_1>0$, we have
    \begin{align}
        \Delta_b|\nabla_b v|^2\geq & \frac{1}{m}(\Delta_b v)^2+4mv^2_0+2|\pi_{1,1}^\perp \nabla_b d_b v|^2
              -2\langle \nabla_b \Delta_b v,
               \nabla_b v \rangle
        \nonumber\\
        & -\epsilon_1|\nabla_b v_0^2|-(2k+16\epsilon_1^{-1})|\nabla_b v|^2 
    \end{align}
     and
     \begin{align}
         \Delta_bv^2_0\geq & 2|\nabla_b v_0|^2-2\langle \nabla_\xi v,\nabla_\xi \Delta_b v\rangle-2k_1 |\pi_{1,1}^\perp \nabla_b d_b v|^2 \nonumber
         \\
         &-4k_1|v_0|^2-2k_1|\nabla_b v|^2,
     \end{align}
     where $|\pi_{1,1}^\perp \nabla_b d_b v|^2=2v_{\alpha\beta}v_{\Bar{\alpha}\Bar{\beta}} $.
 \end{lemma} 
     For our purpose, we will use the Riemannian distance $r$ to construct cut-off function. Hence we introduce a sub-Laplacian comparison theorem.
     \begin{theorem} (\cite{CDRZ})
         Let $(M^{2m+1},HM,J,\theta)$ be a pseudo-Hermitian manifold with
    \begin{align*}
        Ric_b +2(m-2)Tor_b\geq -k\quad and\quad |A|,|\nabla_b A|\leq k_1\nonumber
    \end{align*}
    for some $k,k_1\geq 0$. Then there exists a constant $C$ only depending on $m$ such that
    \begin{align*}
        \Delta_b r\leq C(\frac{1}{r}+\sqrt{1+k+k_1+k_1^2})
    \end{align*}
    outside the cut locus of $x_0$, where $r$ is the Riemannian distance from some point $x$ to a fixed point $O$.
     \end{theorem}

     \section{Proof of Theorem 1.1}
     In this section, we assume that $(M^{2m+1},HM,J,\theta)$ is a pseudo-Hermitian manifold with
     \begin{align*}
        Ric_b +2(m-2)Tor_b\geq -k\quad and\quad |A|,|\nabla_b A|\leq k_1 ,\nonumber
    \end{align*}
    for some constants $k,k_1\geq 0$. In the following calculations, the universal constant $C$ might be changed from line to line. 
    \par
    Let $r(x)$ be the Riemannian distance function from $x\in M$ to a fix point $O$ and $B_R=B_R(O)$ is the Riemannian ball centered at $O$ with radius $R$. In this section, we assume that $R\geq 1$. Suppose that $u$ is a positive solution of (1.3) on $B_{2R}$ and $f(t)=e^{pt}$, then (1.3) is equivalent to the following equation
     \begin{align}
         \Delta_bu+auf(log\ u)=0, \quad on\ B_{2R}.
     \end{align}
     Set $w=log\ u$ and $F=|\nabla_b w|^2+bf$, where $b\geq 0$ is a constant to be determined later, we have
     \begin{align}
         \Delta_b w+F+(a-b)f(w)=0,\quad |\nabla_b w|^2=F-bf(w).
     \end{align} 
      Choose a $ \phi=\varphi(\frac{r}{R})$, where $\varphi$ is a cut-off function such that
    \begin{eqnarray}
        \varphi|_{[0,1]}=1,\ \varphi|_{[2,\infty)}=0,\ -C|\varphi|^{\frac{1}{2}}\leq \varphi^{'}\leq 0,\ \varphi^{''}\geq -C.\nonumber
     \end{eqnarray}
     Furthermore, $\phi$ satisfies that 
     \begin{align*}
         \frac{|\nabla_b \phi|^2}{\phi}\leq \frac{C}{R^2},\quad \Delta_b \phi\geq -\frac{C}{R}.
     \end{align*}
     
     \par
     We define the auxiliary real-valued function by 
     \begin{align*}
         \Phi=F+\mu \phi|w_0|^2=|\nabla_b w|^2+bf(w)+\mu \phi|w_0|^2,
     \end{align*}
     where $\mu>0$ is a constant depends on $R$. From Lemma 2.1, we derive the following estimates.
      \begin{lemma} 
      Suppose $k_1\mu \leq  1$ and $\phi(x),\Phi(x)\neq 0$, then at $x$, we have
    \begin{align}
        \Delta_b \Phi \geq & \frac{1}{m}F^2-2\langle \nabla_b w ,\nabla_b \Phi\rangle -4k_1\mu\phi|\nabla_b w|^2|w_0|+2\mu |w_0|^2\langle \nabla_b \phi,\nabla_b w\rangle  \nonumber
        \\
        &+[(b-2a)f^{'}+bf^{''}-\frac{2}{m}(b-a)f-(2k+16(\mu\phi)^{-1}+2k_1\mu\phi)]F \nonumber
        \\
        &+\frac{(b-a)^2}{m}f^2+(2k+16(\mu\phi)^{-1}+2k_1\mu\phi)bf-b(b-a)f^{'}f-b^2f^{''}f \nonumber
        \\
        &+(4m-4k_1\mu\phi-2a\mu\phi f^{'}+\mu\Delta_b\phi-4\mu\frac{|\nabla_b\phi|^2}{\phi})|w_0|^2.
    \end{align}
     \end{lemma} 
    \begin{proof}
    From (3.2), we have
    \begin{align*}
        \langle \nabla_\xi w,\nabla_\xi \Delta_b w\rangle=& \langle \nabla_\xi w,\nabla_\xi (-|\nabla_b w|^2-af(w))\rangle
        \\
        =&-2\langle w_0\nabla_\xi\nabla_b w,\nabla_b w\rangle-2af^{'}|w_0|^2
        \\
        =&-\langle \nabla_b|w_0|^2,\nabla_b w\rangle+2w_0A(\nabla_bw ,\nabla_bw)-2af^{'}|w_0|^2.
    \end{align*}
    In the last equality, we use the communication relation
    \begin{align*}
        \nabla_\xi \nabla_bw=\nabla_bw_0-\tau(\nabla_b w).
    \end{align*}
        Hence Lemma 2.1 yields that
        \begin{align}
        \Delta_b F\geq & \frac{1}{m}(\Delta_b w)^2+4mw^2_0+2|\pi_{1,1}^\perp \nabla_b d_b w|^2
              -2\langle \nabla_b \Delta_b w,
               \nabla_b w \rangle-\epsilon_1|\nabla_b w_0^2|
        \nonumber\\
        & -(2k+16\epsilon_1^{-1})|\nabla_b w|^2 +b(f^{''}|\nabla_b w|^2+f^{'}\Delta_bw)
    \end{align}
     and
     \begin{align}
         \Delta_bw^2_0\geq & 2|\nabla_b w_0|^2-2\langle \nabla_b|w_0|^2,\nabla_b w\rangle-4k_1|\nabla_b w|^2|w_0|-2k_1 |\pi_{1,1}^\perp \nabla_b d_b w|^2 \nonumber
         \\
         &-(4k_1+2af^{'})|w_0|^2-2k_1|\nabla_b w|^2.
     \end{align}
     Furthermore, we have 
     \begin{align}
         \Delta_b(\phi w^2_0)=&\phi \Delta_b w_0^2+2\langle \nabla_b \phi,\nabla_b w_0^2\rangle +w_0^2\Delta_b \phi \nonumber
         \\
         \geq& 2\langle \nabla_b \phi,\nabla_b w_0^2\rangle +w_0^2\Delta_b \phi+2\phi|\nabla_b w_0|^2-2\langle \nabla_bw,\nabla_b(\phi|w_0^2|)\rangle \nonumber
         \\
         &+2|w_0^2| \langle \nabla_b\phi,\nabla_b w\rangle -4k_1\phi|\nabla_b w|^2|w_0|-2k_1 \phi|\pi_{1,1}^\perp \nabla_b d_b w|^2 \nonumber
         \\
         &-(4k_1+2af^{'})\phi|w_0|^2-2k_1\phi|\nabla_b w|^2.
     \end{align}
     Combining with (3.4) and choosing $\epsilon_1=\mu \phi$, we derived that
     \begin{align}
         \Delta_b \Phi\geq &\frac{1}{m}(F-(b-a)f)^2-2\langle \nabla_bw,\nabla_b \Phi\rangle -2(a-b)f^{'}|\nabla_b w|^2\nonumber
         \\
         &+\mu\phi|\nabla_bw_0|^2+2\mu\langle \nabla_b \phi,\nabla_b|w_0|^2\rangle-4k_1\mu\phi|\nabla_b w|^2|w_0| \nonumber
         \\
         &-(2k+16(\mu\phi)^{-1}+2\mu k_1\phi-bf^{''})|\nabla_b w|^2+2(1-k\mu)|\pi_{1,1}^\perp \nabla_b d_b w|^2 \nonumber
         \\
         &+(4m-4k_1\mu\phi-2a\mu\phi f^{'}+\mu\Delta_b\phi)|w_0^2|+2\mu w_0^2\langle\nabla_b\phi,\nabla_b w\rangle +bf^{'}\Delta_b w.
     \end{align}
     Since $k_1\mu\leq 1$ and using (3.2), the proof is finished by 
     \begin{align}
         \mu\phi|\nabla_bw_0|^2+2\mu\langle \nabla_b \phi,\nabla_b|w_0|^2\rangle \geq -4\mu \frac{|\nabla_b \phi|^2}{\phi}|w_0|^2. \nonumber
     \end{align}
    \end{proof}
     Define $\phi\Phi=P+\mu Q,$ where $ P=\phi F, Q=\phi^2|w_0|^2$. We assume that $x_0$ is the maximum point of $\phi \Phi$ on $ B_{2R}$. Since $\phi \Phi=0$ on the $ \partial B_{2R}$, $x_0$ must lie in $B_{2R}$. Set $ P_0=P(x_0), Q_0=Q(x_0)$ and let $b=\lambda a$.
     
     \begin{lemma}
         If $\phi \Phi(x_0)\neq 0$ and $k_1\mu \leq\epsilon \leq  1$, we have
         \begin{align}
             0\geq &  (\frac{1}{m}-\epsilon)P_0^2+a\phi e^{pw}[(T+\lambda\frac{\epsilon}{2})P_0+Na\phi e^{pw}] \nonumber
             \\
             &-(2k+16\mu^{-1}+2\epsilon+\frac{C}{R})P_0+(4m-12\epsilon-2a\mu\phi f^{'}-\mu\frac{C}{R}-\mu\frac{C}{R}P_0^{\frac{1}{2}})Q_0,
         \end{align}
         where  
         \begin{align*}
             T=\lambda p^2+(\lambda-2)p-\frac{2}{m}(\lambda-1),\quad N=\frac{(\lambda-1)^2}{m}-\lambda(\lambda-1)p-\lambda^2p^2 . 
         \end{align*}
     \end{lemma}
     \begin{proof}
         From Lemma 3.1 and (3.2), we have
         \begin{align*}
             \phi\Delta_b(\phi\Phi)\geq & \frac{1}{m}P^2+(\Delta_b\phi-2\frac{|\nabla_b \phi|^2}{\phi})\phi \Phi+2\langle\nabla_b \phi,\nabla_b (\phi\Phi)\rangle-2\phi\langle\nabla_b w,\nabla_b (\phi\Phi)\rangle
             \\
             &+2\langle \nabla_b w,\nabla_b\phi\rangle \phi\Phi-4k_1\mu \phi^3|\nabla_b w|^2|w_0|+2\mu\phi^2|w_0|^2\langle \nabla_b \phi,\nabla_bw \rangle 
             \\
            &+[(b-2a)f^{'}+bf^{''}-\frac{2}{m}(b-a)f-(2k+16(\mu\phi)^{-1}+2k_1\mu\phi)]\phi P \nonumber
        \\
        &+[\frac{(b-a)^2}{m}f^2+(2k+16(\mu\phi)^{-1}+2k_1\mu\phi)bf-b(b-a)f^{'}f-b^2f^{''}f]\phi^2 \nonumber
        \\
        &+(4m-4k_1\mu\phi-2a\mu\phi f^{'}+\mu\Delta_b\phi-4\mu\frac{|\nabla_b\phi|^2}{\phi})Q.
         \end{align*}
         Note that $x_0$ is the maximum point and $x_0\notin \partial B_{2R}$, we have
         \begin{align*}
             \nabla_b(\phi\Phi)(x_0)=0,\quad \Delta_b (\phi\Phi)(x_0)\leq 0.
         \end{align*}
         By the definition of $\phi$, at $x_0$, we have 
         \begin{align}
             0\geq &  \frac{1}{m}P_0^2-\frac{C}{R}(P_0+\mu Q_0)-4k_1\mu \phi|\nabla_b w|^2\phi |w_0|-2\mu \phi^2|w_0|^2|\nabla_b w|\phi^{\frac{1}{2}}\frac{|\nabla_b \phi|}{\phi^{\frac{1}{2}}} \nonumber
             \\
            &+[(b-2a)f^{'}+bf^{''}-\frac{2}{m}(b-a)f-(2k+16(\mu\phi)^{-1}+2k_1\mu\phi)]\phi P_0 \nonumber
        \\
        &+[\frac{(b-a)^2}{m}f^2+(2k+16(\mu\phi)^{-1}+2k_1\mu\phi)bf-b(b-a)f^{'}f-b^2f^{''}f]\phi^2 \nonumber
        \\
        &+(4m-4k_1\mu\phi-2a\mu\phi f^{'}-\mu\frac{C}{R})Q_0-2|\nabla_b w|^{\frac{1}{2}}|\nabla_b\phi|(P_0+\mu Q_0).
         \end{align}
         Since $P=\phi F\geq \phi|\nabla_b w|^2$ and $R\ge 1$, using Cauchy inequality, we derive that 
         \begin{align*}
             2|\nabla_b w|^{\frac{1}{2}}|\nabla_b\phi|(P_0+\mu Q_0)\leq & 2 |\phi\nabla_bw|^{\frac{1}{2}}\frac{|\nabla_b\phi|}{\phi^{\frac{1}{2}}}(P_0+\mu Q_0)
             \\
             \leq & \frac{C}{R}P_0^{\frac{1}{2} } Q_0+2 \phi^{\frac{1}{2}}|F-bf|^{\frac{1}{2}}\frac{|\nabla_b\phi|}{\phi^{\frac{1}{2}}}P_0
             \\
             \leq &\frac{C}{R}P_0^{\frac{1}{2} } Q_0+\frac{\epsilon}{2}P_0\phi (F-bf),
             \\
             4k_1\mu \phi|\nabla_b w|^2\phi |w_0|\leq & \frac{\epsilon}{2}P_0^2+8\epsilon Q_0,
             \\
             2\mu \phi^2|w_0|^2|\nabla_b w|\phi^{\frac{1}{2}}\frac{|\nabla_b \phi|}{\phi^{\frac{1}{2}}}\leq & \frac{C}{R}P_0^{\frac{1}{2} } Q_0.
         \end{align*}
         Plugging the above into (3.9), we have 
         \begin{align*}
             0\geq &  (\frac{1}{m}-\epsilon)P_0^2+\frac{\epsilon }{2}\phi P_0\lambda a f+[\lambda p^2+(\lambda-2)p-\frac{2}{m}(\lambda-1)]ae^{pw}\phi P_0
             \\
             &-(2k+16\mu^{-1}+2\epsilon+\frac{C}{R})P_0+[\frac{(\lambda-1)^2}{m}-\lambda(\lambda-1)p-\lambda^2p^2]a^2\phi^2e^{2pw}
             \\
             &+(4m-12\epsilon-2a\mu\phi f^{'}-\mu\frac{C}{R}-\mu\frac{C}{R}P_0^{\frac{1}{2}})Q_0.
         \end{align*}
         This completes the proof.
     \end{proof}
     Now we can give the proof of Theorem 1.1.
     
     ~\\
     $\mathbf{Proof\ of\ Theorem\ 1.1.}$ 
     \par
     As above, we first consider the equation (1.3) on $B_{2R}$. Suppose that the maximum value $\phi\Phi(x_0)>0$, otherwise the conclusion is trivial. According to the sign of $a$ and $p$, we study the following two cases.

     ~\\
     $\mathbf{Case\ 1}$: $a>0$ and $p\leq 0$. In this case, we choose $\lambda=1$. Hence 
     \begin{align*}
        b=a,\quad T=p^2-p,\quad N=-p^2, \quad 2a\mu\phi f^{'}\leq 0.
     \end{align*}
     Since $N\leq 0 $ and $ P_0\geq \phi ae^{pw}$, we find that $Na\phi e^{pw}\geq N P_0.$ Moreover, 
     \begin{align*}
         (T+\frac{\epsilon}{2})P_0+Na\phi e^{pw}\geq (T+\frac{\epsilon}{2}+N )P_0=(\frac{\epsilon}{2}-p)P_0\geq 0.
     \end{align*}
     Therefore, by Lemma 3.2, we have
         \begin{align}
             0\geq &  (\frac{1}{m}-\epsilon)P_0^2 
            -(2k+16\mu^{-1}+2\epsilon+\frac{C}{R})P_0+(4m-12\epsilon-\mu\frac{C}{R}-\mu\frac{C}{R}P_0^{\frac{1}{2}})Q_0.
         \end{align}
    Let 
    \begin{align*}
        q=\max\limits_{B_{2R} } P \geq P_0, \quad \epsilon=\frac{1}{6m},\quad \mu^{-1}=6m k_1+1+2(\frac{C}{R}+\frac{C}{R}q^{\frac{1}{2}}).
    \end{align*}
    This yields that 
    \begin{align*}
        4m-12\epsilon-\mu\frac{C}{R}-\mu\frac{C}{R}P_0^{\frac{1}{2}}\geq m, \quad k_1\mu \leq \epsilon.
    \end{align*}
    Then (3.10) induces that 
    \begin{align*}
        0\geq \frac{5}{6m}P_0^2-C(1+k+k_1+\frac{C}{R}+\frac{C}{R}q^{\frac{1}{2}})P_0+mQ_0.
    \end{align*}
    Then we can obtain 
    \begin{align}
        P_0\leq & C(1+k+k_1+\frac{C}{R}+\frac{C}{R}q^{\frac{1}{2}}),
        \\ 
        Q_0\leq & C(1+k+k_1+\frac{C}{R}+\frac{C}{R}q^{\frac{1}{2}})^2.
    \end{align}
    Since the definition of $\mu $ yields that 
    \begin{align*}
        \mu (1+k+k_1+\frac{C}{R}+\frac{C}{R}q^{\frac{1}{2}})\leq C,
    \end{align*}
    we have 
    \begin{align*}
        q=\max\limits_{B_{2R} } P\leq \phi\Phi(x_0)=P_0+\mu Q_0\leq C(1+k+k_1+\frac{C}{R}+\frac{C}{R}q^{\frac{1}{2}}).
    \end{align*}
    In particular, we find that 
    \begin{align*}
        q\leq \frac{C}{R}+\frac{C}{R}q^{\frac{1}{2}}.
    \end{align*}
    Hence for any $R\geq 1$, we have 
    \begin{align*}
        \max\limits_{B_{2R} } P\leq C, \quad \max\limits_{B_{2R} } Q\leq C,
    \end{align*}
    where $C$ depends on $m,k_1,k.$ Therefore we can conclude that $ P, Q$ are actually uniformly bounded on $M$. Now we can choose $\epsilon\in (0,\frac{1}{6m}), s>0$ and define 
   \begin{align*}
       \mu^{-1}=k_1\epsilon^{-1}+s+\frac{C}{R}
   \end{align*}
   such that
   \begin{align*}
        4m-12\epsilon-\mu\frac{C}{R}-\mu\frac{C}{R}P_0^{\frac{1}{2}}\geq m, \quad k_1\mu \leq \epsilon.
    \end{align*}
    Then (3.10) becomes 
     \begin{align}
             0\geq &  (\frac{1}{m}-\epsilon)P_0^2 
            -(\lambda_s+\frac{C}{R})P_0+mQ_0,
         \end{align}
         where $\lambda_s=2k+2\epsilon+16s+16k_1\epsilon^{-1}.$ Hence (3.13) implies that 
         \begin{align*}
             P_0\leq  \frac{\lambda_s+\frac{C}{R}}{\frac{1}{m}-\epsilon},\quad 
             Q_0\leq  \frac{(\lambda_s+\frac{C}{R})^2}{4m(\frac{1}{m}-\epsilon)}.
         \end{align*}
         Then for $a>0,p\leq 0$, there holds
         \begin{align*}
             \max\limits_{B_R} (P+\mu Q)\leq \frac{\lambda_s+\frac{C}{R}}{\frac{1}{m}-\epsilon}+\mu \frac{(\lambda_s+\frac{C}{R})^2}{4m(\frac{1}{m}-\epsilon)}.
         \end{align*}
         Letting $R\xrightarrow{} \infty, $ we obtain (1.4).
         
     ~\\
     $\mathbf{Case\ 2}$: (1) $a<0,p>0$. In this case, we choose $\lambda=0.$ Hence
     \begin{align*}
         b=0,\quad T=\frac{2}{m}-2p, \quad N=\frac{1}{m},\quad 2a\mu \phi f^{'}\leq 0.
     \end{align*}
     For $p\geq \frac{1}{m}$, we have $T\leq 0$. Note that $a<0$, hence
     \begin{align*}
         a\phi e^{pw}(TP_0+Na\phi e^{pw})\geq 0.
     \end{align*}
     It follows that 
     \begin{align}
             0\geq &  (\frac{1}{m}-\epsilon)P_0^2 
            -(2k+16\mu^{-1}+2\epsilon+\frac{C}{R})P_0+(4m-12\epsilon-\mu\frac{C}{R}-\mu\frac{C}{R}P_0^{\frac{1}{2}})Q_0,
         \end{align}
    which is the same as in the case $a>0,p\leq 0$. By the same method, we obtain (1.5).
    \par
    For $0< p < \frac{1}{m}$, we have $ T>0$. By Cauchy inequality, we get
    \begin{align*}
         a\phi e^{pw}(TP_0+Na\phi e^{pw})= (\frac{1}{\sqrt{N}}a\phi e^{pw})^2+a\phi e^{pw}TP_0\geq -\frac{T^2}{4N}P^2.
    \end{align*}
    Observe that 
    \begin{align*}
       \frac{T^2}{4N}=m(\frac{1}{m}-p)^2<\frac{1}{m},
    \end{align*}
    hence we have the following
    \begin{align}
             0\geq &  (\frac{1}{m}-\epsilon-\frac{T^2}{4N})P_0^2 
            -(2k+16\mu^{-1}+2\epsilon+\frac{C}{R})P_0 \nonumber
            \\
            &+(4m-12\epsilon-\mu\frac{C}{R}-\mu\frac{C}{R}P_0^{\frac{1}{2}})Q_0, \nonumber
         \end{align}
        by requiring that $ \epsilon< \frac{1}{m}-\frac{T^2}{4N}$. Using the same method in the case $a>0,p\leq 0$, we obtain (1.6).

       \section{ Acknowledgments}
        The author would like to thank Professor Yuxin Dong and Professor Xiaohua Zhu for their continued support and encouragement.

\bibliographystyle{siam}
\bibliography{ref}

\end{document}